\documentclass[a4paper, 11pt, reqno]{amsart}

\usepackage{amsmath, amssymb, amsthm, comment, hyperref}
\usepackage{mathtools}
\mathtoolsset{showonlyrefs}
\numberwithin{equation}{section}

\theoremstyle{plain}
\newtheorem{theorem}{Theorem}[section]
\newtheorem{proposition}[theorem]{Proposition}
\newtheorem{lemma}[theorem]{Lemma}

\theoremstyle{definition}
\newtheorem{definition}[theorem]{Definition}

\theoremstyle{remark}
\newtheorem{remark}[theorem]{Remark}

\title{Proximal algorithm and calibrated cycles}
\author{Ryohei Chihara}
\address{}
\email{}
\subjclass[2010]{53C38, 90C25, 49J40}
\keywords{calibration, proximal algorithm, total variation}

\begin{document}
\maketitle

\begin{abstract}
We sketch an application of proximal algorithms to the deformation of de Rham currents into cycles, which is presented as a convex optimization problem. Emphasis is placed on the use of total variation denoising for differential forms, specifically in constructing calibrated cycles in calibrated manifolds.
\end{abstract}

% Main Matter
\section{Introduction}

In this note, we explore the question: How can we transform a de Rham current into a cycle that is located as closely as possible to it? Our approach focuses on the concept of denoising. In our context, we can compare a given de Rham current, denoted as $T_0$, to a noisy image. Here, the boundary $\partial T_0$ represents noise, drawing a direct analogy to the study of total variation denoising as discussed in \cite{ROF}. Our goal is to remove this noise and find a cycle $T$.  We look at a convex optimization problem to do this, with an emphasis on making calibrated cycles. We outline some results for this problem, obtained by applying the proximal algorithm, often without rigorous proofs.

This note is structured as follows: Section 2 explains the basics of exterior algebras, de Rham currents and the idea of positivity. In Section 3, we explain our proximal algorithm and show how it works in an unconstrained setting. Section 4 presents an application of our method in calibrated geometry.

\section{Preliminaries}

\subsection{Exterior algebra} 
Let $V$ be an $n$-dimensional real vector space with a metric and an orientation. 
Let $\wedge^{k}V$ denote the subspace that consists of all $k$-vectors in the Exterior algebra $\oplus_{k}\wedge^{k}V$. A $k$-vector $v$ can be interpreted as an alternative multi-linear $k$-form over the dual vector space $V^{*}$. 
The metric on $V$ induces two natural norms for $v \in \wedge^{k}V$, the \textit{Euclidean norm} $|v|$ and the \textit{mass norm} $\|v\|$.
Although the Euclidean norm is the same as the Euclidean norm for tensor algebras up to a constant, the mass norm is defined in a more intrinsic manner.
A $k$-vector $v$ is called to be \textit{decomposable} if there are some $v_1, v_2, \ldots, v_k \in V$ such that $v = v_1 \wedge v_2 \wedge \dots \wedge v_k$. We can identify the set of all unit decomposable $k$-vectors with that of all oriented $k$-planes, denoted by $G(k, n)$. 
The mass norm $\|v\|$ for $v \in \wedge^{k}V$ is defined such that the unit ball $\{v \in \wedge^{k}V \ |\  \|v\| \leq 1\}$ coincides with the convex hull of $G(k, n)$. The following lemma gives another characterization of the mass norm.

\begin{lemma}\label{lem: mass}
   For any $k$-vector $v$, there are some decomposable $k$-vectors $b_1, b_2, \ldots, b_l$ such that $v = b_1 + b_2 + \dots + b_l$ and $\|v\| = |b_1| + |b_2| + \dots  + |b_l|$. Moreover, $\|v\| \leq |c_1| + |c_2| + \dots + |c_m|$ for any decomposable $k$-vectors $c_1, c_2, \ldots, c_m$ satisfying $v = c_1 + c_2 + \dots + c_m$.
\end{lemma}

\begin{proof}
   For any $v \neq 0$, $\|v\| = \inf\{ c > 0 \ |\ v/c \in \text{the convex hull of }G(k,n)\}$. By the compactness of $G(k, n)$ and its convex hull, we can take a decreasing sequence $\{c_i\}$ such that $c_i$ converges to $\|v\|$ and $v/c_i$ converges to $a_1b_1 + a_2b_2 + \dots + a_lb_l$ where $a_i > 0$ and $b_i \in G(k, n)$ for all $i = 1, 2, \ldots, l$ and $a_1 + a_2 + \dots + a_l = 1$. Then the limit $v / \|v\|= a_1b_1 + a_2b_2 + \dots + a_lb_l$ gives the desired sum. The later is deduced from the fact that $|c_1| + |c_2| + \dots + |c_m| \in \{ c > 0 \ |\ v/c \in \text{the convex hull of }G(k,n)\}$ holds.
\end{proof}

The \textit{comass norm} for $v \in \wedge^{k}V$ is defined by $\|v\|^{*} = \sup \{ \langle v, w \rangle \ |\ w \in \wedge^{k}V \text{ and } \|w\| \leq 1\}$, where $\langle v, w \rangle$ is the inner product determined by the standard norm for $\wedge^{k}V$. 
The relation of these three norms is as follows.

\begin{proposition}
For any $v \in \wedge^{k}V$, we have $\|v\|^{*} \leq |v| \leq \|v\|$. Moreover, the following statements are equivalent: (i) $v$ is decomposable, (ii) $|v| = \|v\|$, (iii) $\|v\|^{*} = |v|$.
\end{proposition}

\begin{proof}
Since the oriented Grassmannians $G(k,n) \subset \wedge^{k}V$ is compact, there exists a $w \in G(k, n)$ such that $\|v\|^{*} = \langle v, w \rangle$. Therefore, $\|v\|^{*} \leq |v|$. By Lemma \ref{lem: mass}, we can express $v$ as a sum of decomposable k-vectors $b_1, b_2, \ldots, b_l$ such that $\|v\| = |b_1| + |b_2| + \dots + |b_l|$. Using the Schwarz lemma, we obtain
\begin{align}
|v|^{2} &= |b_1|^{2} + |b_2|^{2} + \dots + |b_l|^{2} + 2\sum_{i,j}\langle b_i, b_j \rangle \\
        &\leq |b_1|^{2} + |b_2|^{2} + \dots + |b_l|^{2} + 2\sum_{i,j}|b_i||b_j| \\
        &= \|b\|^{2}.
\end{align}
Hence, $|v| \leq \|v\|$, proving the first part of the proposition. The second part can be proved in a similar manner.
\end{proof}

\subsection{Function spaces}
\label{subsec: func}
Let $M$ be a compact smooth $n$-dimensional manifold without boundary. We assume that $M$ has a metric $g$ and an orientation.
Let $C^{0}(M, \wedge^{k}T^{*}M)$ be the Banach space of all continuous sections of $k$-forms over $M$ with $\|\rho\|_{C^{0}(M)} = \sup\{|\rho(x)| \ |\ x \in M\}$. Then the dual space $C^{0}(M, \wedge^{k}T^{*}M)^{*}$ is also a Banach space, denoted by $B$. We call $T \in B$ a \textit{current} since we can see $T$ as a de Rham current. For example, $B$ includes compact $k$-dimensional submanifolds of $M$.
\begin{comment}
\begin{proposition}
\label{prop: rep}
Let $T \in B$. Then there exist a measurable section $\omega$ of $\wedge^{(n-k)}T^{*}M$ and a Borel measure $\mu$ such that (i) $|\omega|$ is integrable with respect to $\mu$, (ii) for any $\rho \in C^{0}(M, \wedge^{k}T^{*}M)$, we have
    \begin{align}
    T(\rho) = \int_{M}\frac{\rho \wedge \omega}{\mathrm{vol}} d\mu
    \end{align}
where $\rho \wedge \omega$ is the wedge product, and $\mathrm{vol}$ is the unit volume form on $M$. 
\end{proposition}

\begin{proof}

\end{proof}
\end{comment}
Let $H$ denote the Hilbert space  $L^{2}$-sections $L^{2}(M, \wedge^{(n-k)}T^{*}M)$ of all $L^{2}$-sections of $\wedge^{(n-k)}T^{*}M$. Then $\omega \in H$ defines $T_{\omega} \in B$ by
\begin{align}
    T_{\omega}(\rho) = \int_{M}{\rho \wedge \omega} 
\end{align}
for any $\rho \in C^{0}(M, \wedge^{k}T^{*}M)$. So we can see $H$ as a subspace of $B$. Moreover, $H$ is dense in the sense of weak$^{*}$ topology on $B$.  If $\omega \in H$ is smooth, then $\partial T_{\omega} = 0$ as a de Rham current if and only if $d\omega = 0$, where $\partial$ is the boundary operator and $d$ is the exterior derivation. In this note we call a current $T$ satisfying $\partial T = 0$ a {\it cycle}, and moreover if $T=T_{\omega}$ for some $\omega \in H$ then we call $T$ an {\it $L^2$ cycle}.

\begin{comment}
Our focus lies on exploring geometric objects within the closed set $\bar{U}$, including submanifolds and differential forms.
These objects can be studied in the framework of functional analysis using the concept of de Rham currents. Although de Rham currents are defined as a generalization of Schwartz distributions for smooth differential forms, we restrict our interest to Banach spaces, which allows for a more elementary discussion.
\end{comment}

\subsection{Positivity}
Let us review \textit{positive currents}, which were defined as a generalization of Radon measures in \cite{S}. They include calibrated submanifolds and more generally calibrated currents, introduced in \cite{HL}.

\begin{definition}
Let $V$ be a topological vector space. A convex cone $L \subset V$ is called a \textit{compact convex cone} if there is a $\phi \in V^{*}$ such that (i) $\phi(v) > 0$ for all $v \neq 0$ in $L$ and (ii) $\{v \in V \ |\ \phi(v) = 1\} \cap L$ is compact.
\end{definition}

\begin{definition}
Let $M$ be a smooth manifold. A continuous subbundle $\Lambda_{+}$ of $\wedge^{k}T^{*}M$ is called a \textit{cone structure} on $M$ if it satisfies that (i) the support of $\Lambda_{+}$ is compact and (ii) $\Lambda_{+}(x) \subset \wedge^{k}T^{*}M_{x}$ is a compact convex cone for all $x \in M$.  
\end{definition}

From now on, we assume that $M$ is a closed smooth $n$-dimensional manifold with a metric and an orientation. Let $B$ and $H$ denote the dual space $C^{0}(M, \wedge^{k}T^{*}M)^{*}$  and  the $L^{2}$ space $L^{2}(M, \wedge^{(n-k)}T^{*}M)$ as in \$\ref{subsec: func}, respectively. Let $\Lambda_{+} \subset \wedge^{(n-k)}T^{*}M$ be a cone structure on $M$.

\begin{definition}
A current $T \in B$ is called a \textit{positive current} if there exists a sequence $\{\omega_{i}\}$ in $H$ such that (i) $\omega_{i}$ converges to $T$ in the sense of weak$^{*}$ topology and (ii) $\omega_{i}(x) \in \Lambda_{+}(x)$ for all $i \in \mathbb{N}$ and $x \in M$.
\end{definition}

Let C denote the set of all positive currents in $B$. The following is an analogy of weak$^*$ compactness of Radon probability measures.

\begin{proposition}
\label{prop: compactness}
The subset $C$ of $B$ is a compact convex cone in the sense of the weak$^{*}$ topology.
\end{proposition}

\begin{proof}
By definition, we can see that $C$ is a convex cone. Using partition of unity, we can construct a continuous $k$-form $\phi$ on $M$ such that $(\phi(x) \wedge \omega) / \mathrm{vol} > 0$ for all $x \in M$ and $\omega \neq 0$ in $\Lambda_{+}(x)$. By the compactness of the support of $\Lambda_{x}$, there is a positive constant $k$ such that $(1/k)|\omega| \leq (\phi \wedge \omega) / \mathrm{vol} \leq k|\omega|$ for all $\omega \in \Lambda_{+}$. Thus, $(1/k)\int_{M}|\omega|\mathrm{vol} \leq \int_{M}\phi \wedge \omega = T_{\omega}(\phi) \leq k\int_{M}|\omega|\mathrm{vol}$ for any $\omega \in C \cap H$. From the properties of weak$^*$ convergence we can prove that $\phi$ is a transversal form, and by the Banach-Alaoglu theorem we can prove the compactness of $\{T \in 
 B \ |\ T(\phi) = 1\} \cap C$. 
\end{proof}
We can easily apply the above general discussion to calibrated geometry.
\begin{definition}
A smooth $k$-form $\phi$ on $M$ is called a \textit{calibration} if it satisfies $\|\phi(x)\|^{*} \leq 1$ for all $x \in M$ and $d\phi = 0$. 
\end{definition}

Let $\Lambda_{\phi}(x)$ = $\{\omega \in \wedge^{(n-k)}T^{*}_{x}M\ |\ \phi \wedge \omega = \|\omega\|\mathrm{vol}\}$ for $x \in M$. Then $\Lambda_{\phi}(x)$ is a compact convex cone in $\wedge^{(n-k)}T^{*}_{x}M$ and $\Lambda_{\phi} \subset \wedge^{(n-k)}T^{*}M$ is a cone structure on $M$. We call a positive current $T \in B$ for $\Lambda_{\phi}$ a \textit{calibrated current} and denote the set of all calibrated currents by $C_{\phi}$. 
The following is straightforward from Proposition \ref{prop: compactness}.
\begin{proposition}\label{prop: calcomp}
The cone $C_{\phi}$ in $B$ is a compact convex cone in the sense of the weak$^{*}$ topology. In particular, the calibration $phi$ is a transversal form for $C_{\phi}$ and $\{T \in B\ |\ T(\phi) = 1\} \cap C_{\phi}$ is compact.
\end{proposition}

We call a calibrated current $T$ satisfying $\partial T = 0$ a {\it calibrated cycle} and moreover if it is defined by an $L^2$ differential form, we call it an {\it $L^2$ calibrated cycle}.

\section{Total variation denoising}

\subsection{Proximal algorithm}
A standard method for solving convex optimization problems is the proximal algorithm, established in \cite{R}. In this section, we apply this method to a toy model. This model is a variant of the heat equation for $L^2$ differential forms, as discussed in \cite{MR, MR2}.

Let $H$ be a Hilbert space and $f:H \to (-\infty, +\infty]$ be a non-constant lower semi-continuous function. Let $z_1 \in H$ and $c > 0$. Then the proximal algorithm defines a sequence $\{z_k\}$ $(k=1,2, \ldots)$ as follows.
\begin{align}
    z_{k+1} \in \mathrm{argmin}_{z\in H}\left(f(z) + \frac{1}{2c}\|z - z_k\|^2\right).
\end{align}
We expect that $\{z_k\}_k$ converges to a minimum point $\{z_{\infty}\}$ of $f(z)$ that is closest to $z_1$ in some sense.

\subsection{TV denoising for differential forms}
Let $M$ be a closed oriented Riemannian $n$-manifold. Let $H$ denote the Hilbert space $L^2(M, \wedge^{p}T^{*}M)$ consisting of $L^2$ differential forms of degree $p$. 

\begin{definition}
Let $\omega \in H$. The {\it total variation energy} $E:H \to [0, \infty]$ of $\omega$ is defined by the $L^1$ energy
\begin{align}
E(\omega) = \int_{M}|d\omega|.
\end{align}
The left-hand side is precisely defined by the $\sup$ of 
\begin{align}
\left\{\int_{M}d\gamma \wedge \omega \ |\ \gamma \in C^{\infty}(M, \wedge^{n-p-1}T^{*}M),\ |\gamma(x)| \leq 1 \text{ for all } x \in M\right\}
\end{align}
where $|\gamma|$ denotes the standard norm for $\gamma \in \wedge^{n-p-1}T^*M$.
\end{definition}

By the definition of $E(\omega)$, we have the following.

\begin{proposition}
The functional $E:H \to [0, \infty]$ is non-constant, convex and lower semi-continuous with respect to the weak topology of $H$.
\end{proposition}

The propositions in the rest of this section are proven using the Hodge decomposition theorem:

\begin{theorem}We have the orthogonal decomposition:
\begin{align}
L^2(M, \wedge^pT^*M) = \overline{d\Omega^{p-1}} \oplus \overline{\delta\Omega^{p+1}} \oplus \mathcal{H}^{p}.
\end{align}
\end{theorem}

By calculating the subdifferential of $E$, we have the following.
\begin{proposition}
Let $\omega \in H$. Then $E(\omega) = 0$ if and only if $\omega \in \overline{d\Omega^{p-1}} \oplus \mathcal{H}^{p}$.
\end{proposition}

\begin{proposition}
Let $\omega \in H$ and $\eta \in \overline{d\Omega^{p-1}} \oplus \mathcal{H}^{p}$. Then we have $E(\omega) = E(\omega + \eta)$.
\end{proposition}

Let $\omega_1 \in H$ and $h>0$. Let us consider the proximal algorithm
\begin{align}
    \omega_{k+1} \in \mathrm{argmin}_{\omega \in H}\left(E(\omega) + \frac{1}{2h}\|\omega - \omega_k\|^2\right).
\end{align}    

Then we observe analogous results to the heat flow for differential forms, as studied in \cite{MR, MR2}.

\begin{proposition}
Let $\eta \in \overline{d\Omega^{p-1}} \oplus \mathcal{H}^{p}$. Then the quantity 
\begin{align}
(\eta, \omega_k) = \int_M\langle \eta, \omega_k \rangle \mathrm{vol} 
\end{align}
is constant for the sequence $\{\omega_k\}_k$.
\end{proposition}

\begin{theorem}
Let $\omega_1 \in H$. The sequence $\{\omega_k\}$ $(k=1, 2, \ldots)$ weakly converges to $\omega_{\infty} \in H$. The limit $\omega_{\infty}$ is the projection of $\omega_1$ to $\overline{d\Omega^{p-1}} \oplus \mathcal{H}^{p}$ and satisfies $E(\omega_{\infty}) = 0$. Then we have 
$\|\omega_1 - \omega_{\infty}\| = \inf_{\omega \in H}\{\|\omega_1 - \omega\|\ |\ E(\omega) = 0\}$.
\end{theorem}

\section{Calibrated cycles}
Let $(M, \phi)$ be an oriented Riemannian manifold with a calibration $\phi$ of degree $k$. Let $H$ be the Hilbert space of $L^2$ $(n-k)$-forms on $M$. Now we define the feasible region of our convex optimization problem by 
\begin{align}
C = \{ \omega \in H \ |\ T_{\omega} \text{ is a calibrated current}\},
\end{align}
where $T_{\omega}$ is a $k$-dimensional current defined by 
\begin{align}
T_{\omega}(\rho) = \int_M\rho \wedge \omega \quad \text{for} \quad \rho \in C^0(M, \wedge^{k}T^{*}M).
\end{align}
Our problem is to minimize 
\begin{align}
E(\omega) = \int_M|d\omega|
\end{align}
via the following proximal algorithm
\begin{align}
\omega_{k+1} \in \mathrm{argmin}_{\omega \in C}\left(E(\omega) + \frac{1}{2h}\|\omega - \omega_k\|^2\right)
\end{align}
where $h$ is a positive constant. Then we can see the following results about the minimizing sequence $\{\omega_k\}_k$. The proofs is given by the variational inequalities for convex cones, see for example \cite{R}.
\begin{proposition}
Let $\omega \in H$ be a calibrated current satisfying $E(\omega)=0$. Then the sequence $\{(\omega, \omega_k)\}_k$ is monotonically increasing, where 
\begin{align}
(\omega, \omega_k) = \int_M \langle \omega, \omega_k\rangle \mathrm{vol}_{M}.
\end{align}
\end{proposition}

Using Theorem 1 in \cite{R}, we can see the following convergence result.

\begin{theorem}
If we have an $L^2$ calibrated cycle $\omega$ and an initial $L^2$ calibrated current $\omega_1$ such that $\int \langle \omega, \omega_1 \rangle \mathrm{vol}_M > 0$ then the minimizing sequence $\{\omega_k\}_k$ weakly converges to a non-zero $L^2$ calibrated cycle $\omega_{\infty}$.
\end{theorem}

\begin{remark}
Note that if $(M, \omega)$ is a compact K\"ahler manifold, then the K\"ahler form $\omega$ becomes an $L^2$ calibrated cycle. In this case, any non-zero positive current weakly converges to a non-zero $L^2$ calibrated cycle. Another example of $L^2$ calibrated cycles comes from calibrated foliations on calibrated manifolds. On the other hand, if a calibrated manifold $(M, \phi)$ has no $L^2$ calibrated cycles, then any minimizing sequence $\{\omega_k\}_k$ converges to zero. However, if we impose the additional constraint $\int_M\phi \wedge \omega = 1$, Proposition \ref{prop: calcomp} indicates that a subsequence of $\{\omega_k\}_k$ weakly converges to a non-zero calibrated current $T$, which might not be in $L^2$ or closed.

\end{remark}

% Bibliography
\bibliographystyle{amsplain}
\bibliography{references}

\end{document}